\documentclass[11pt]{article}
\usepackage[a4paper,margin=2cm]{geometry}
\usepackage[english]{babel}
\usepackage{latexsym,amsmath,enumerate,graphics,enumerate,amsthm,tikz,hyperref,float}
\usepackage[mathscr]{euscript}
\usepackage[affil-it]{authblk}
\usepackage{enumerate,amsthm,dsfont,pstricks}
\usepackage{latexsym,amsmath,amssymb,amscd,wrapfig,graphicx}
\usepackage{hyperref,cleveref}

\newtheorem{theorem}{Theorem}[section]
\newtheorem{lemma}[theorem]{Lemma}
\newtheorem{proposition}[theorem]{Proposition}
\newtheorem{corollary}[theorem]{Corollary}

\theoremstyle{definition}

\newtheorem{example}[theorem]{Example}

\theoremstyle{remark}
\newtheorem{remark}[theorem]{Remark}

\newtheorem*{theorem*}{{\bf Theorem}}

\newtheorem*{assumption*}{{\bf Assumption}}

\let\phi=\varphi
\def\N{\mathbb{N}}

\newcommand{\comment}[1]{}

\numberwithin{equation}{section}

\let\epsilon=\varepsilon

\makeatletter
\def\@maketitle{%
  \newpage
  \null
  \vskip 2em%
  \begin{center}%
  \let \footnote \thanks
    {\Large\bfseries \@title \par}%
    \vskip 1.5em%
    {\normalsize
      \lineskip .5em%
      \begin{tabular}[t]{c}%
        \@author
      \end{tabular}\par}%
    \vskip 1em%
    {\normalsize \@date}%
  \end{center}%
  \par
  \vskip 1.5em}
\makeatother

\begin{document}

\title{\sc \huge Prime ideals and Noetherian properties in vector lattices}

\author{Marko Kandi\'{c}%
\thanks{Email: \texttt{marko.kandic@fmf.uni-lj.si}}}
\affil{Faculty of Mathematics and Physics, University of Ljubljana, Jadranska 19, SI-1000 Ljubljana, Slovenia}
\affil{Institute of Mathematics, Physics and Mechanics, Jadranska 19, SI-1000 Ljubljana, Slovenia}

\author{Mark Roelands%
\thanks{Email: \texttt{m.roelands@math.leidenuniv.nl}}}
\affil{Mathematical Institute, Leiden University, 2300 RA Leiden,
The Netherlands}

\maketitle
\date{}
%\date{\today}

\begin{abstract}
In this paper we study the set of prime ideals in vector lattices and how the properties of the prime ideals structure the vector lattice in question. The different properties that will be considered are firstly, that all or none of the prime ideals are order dense, secondly, that there are only finitely many prime ideals, thirdly, that every prime ideal is principal, and lastly, that every ascending chain of prime ideals is stationary (a property that we refer to as prime Noetherian). We also completely characterize the prime ideals in vector lattices of piecewise polynomials, which turns out to be an interesting class of vector lattices for studying principal prime ideals and ascending chains of prime ideals.
\end{abstract}

{\small {\bf Keywords:} } Vector lattices, prime ideals, principal prime ideals, ascending chains of prime ideals.

{\small {\bf Subject Classification:} } Primary 46A40; Secondary 46B40.

\section{Introduction}
Prime ideals play an important role in the study of vector lattices. The prime ideals are precisely the ideals for which the quotient vector lattices are linearly ordered, and every vector lattice contains an abundance of prime ideals. Indeed, for any element $f$ in a vector lattice, the collection of ideals that do not contain $f$ has a maximal element with respect to set inclusion by Zorn's lemma, and this maximal element is a prime ideal. Furthermore, any ideal containing a prime ideal is again a prime ideal, and the set of prime ideals containing a fixed prime ideal is linearly ordered.  The set of prime ideals in a vector lattice, and certain subsets herein, also yield various representations of  vector lattices. For instance, Yosida proved in \cite{Yosida} that an Archimedean vector lattice is isomorphic to the vector lattice of extended functions on the set of prime ideals that avoid a maximal disjoint set in the underlying vector lattice. Johnson and Kist proved a more general result in \cite{Johnson-Kist} which states that every Archimedean vector lattice is isomorphic to the extended functions on a subset of prime ideals that have zero intersection and avoid a maximal disjoint set in the underlying vector lattice. Properties of the prime ideals in a vector lattice can also encode the structure of the vector lattice. For example, in $C(K)$, the continuous functions on a compact Hausdorff space $K$, it follows that every prime ideal is maximal if and only if $K$ is finite, and more generally, in a uniformly complete vector lattice every prime ideal is maximal if and only if the vector lattice is isomorphic to $c_{00}(\Omega)$ for some set $\Omega$. See \cite{Zaanen} for details. The goal of this paper is to further study the structure of vector lattices in terms of the properties of its prime ideals. We characterize the vector lattices for which the set of prime ideals is specialized to having the following properties. Firstly, that all or none of the prime ideals are order dense, secondly, that there are only finitely many prime ideals, thirdly, that every prime ideal is principal, and lastly, that ascending chains of prime ideals 
are stationary.    

The connection between order dense prime ideals and atoms in vector lattices is a dichotomous one. Namely, every prime ideal is order dense if and only if the vector lattice is atomless, and that none of the prime ideals are order dense if and only if the vector lattice is of the form $c_{00}(\Omega)$ for some set $\Omega$.

As for vector lattices with only finitely many prime ideals, we prove that these vector lattices must be finite-dimensional.

The specialization to all prime ideals being principal is inspired by Cohen's theorem for commutative rings, which says that a commutative ring is Noetherian if and only if every prime ideal is finitely generated, see \cite[Theorem~2]{Cohen50}. Since in vector lattices an ideal is finitely generated if and only if it is principal, the restriction to principal prime ideals is taken for this reason. Cohen's theorem is connected to another classical result from commutative ring theory, Kaplansky's theorem. Kaplansky proved in \cite[Theorem~12.3]{Kaplansky49} that in a Noetherian commutative ring every ideal is principal if and only if every maximal ideal is principal. Combining the two theorems yields that a commutative ring is a principal ideal ring if and only if every prime ideal is principal, which is referred to as the Cohen-Kaplansky theorem. It turns out that a vector lattice is Noetherian if and only if it is finite-dimensional. Hence, we prove the vector lattice analogue of the Cohen-Kaplansky theorem stating that every prime ideal in an Archimedean vector lattice is principal if and only if it is finite-dimensional, which is further equivalent with all ideals being principal. We prove that in a uniformly complete Archimedean vector lattice there are no non-maximal principal prime ideals, and we use this result to prove a vector lattice analogue of the Cohen-Kaplansky theorem for uniformly complete vector lattices. More precisely, every ideal in a uniformly complete Archimedean vector lattice is principal if and only if every maximal ideal is principal, which in turn is equivalent with the vector lattice being finite-dimensional. 

Since Noetherian vector lattices are finite-dimensional and, as mentioned above, there are naturally occurring chains of prime ideals in vector lattices, studying vector lattices in which every ascending chain of prime ideals $P_1\subseteq P_2\subseteq\dots$ is stationary is more interesting. We propose to call these vector lattices \emph{prime Noetherian}.  The uniformly complete Archimedean prime Noetherian vector lattices are completely characterized, and these characterizations depend on the existence of a strong unit. It turns out that a prime Noetherian vector lattice with a strong unit is finite-dimensional and in general, a prime Noetherian vector lattice is isomorphic to $c_{00}(\Omega)$ for some set $\Omega$.     

The vector lattices of piecewise polynomials are a resourceful class of non-uniformly complete vector lattices when studying principal prime ideals and prime Noetherian properties. In these vector lattices we can explicitly construct prime Noetherian vector lattices with prescribed finite maximal lengths of ascending chains of prime ideals. Furthermore, we can construct chains of non-maximal principal prime ideals of prescribed length in these vector lattices, which further shows the significant role uniform completeness plays when studying principal prime ideals in vector lattices.    

The structure of our paper is as follows. Section 2 is the preliminary section where the basics about vector lattices are covered. In Section 3 we investigate the relation between order dense prime ideals and atoms in vector lattices. The main result of this section is the dichotomous connection between the two. In Section 4 we characterize vector lattices with only finitely many prime ideals. Vector lattices in which every prime ideal is principal are studied in Section 5 and the main results are \Cref{T: Cohen}, the vector lattice analogue of the Cohen-Kaplansky theorem for commutative rings, and \Cref{Main Cohen}, the uniformly complete vector lattice analogue of the Cohen-Kaplansky theorem for commutative rings. Prime Noetherian vector lattices are studied in Section 6. The main result of this section characterizes the uniformly complete Archimedean vector lattices with this property in \Cref{T:prime Noetherian uniformly complete}. In Section 7 we study vector lattices of piecewise polynomials.

\section{Preliminaries}

Let $E$ be a vector lattice. We say that $E$ is \emph{Archimedean} whenever it follows from $0\leq nx\leq y$ for each $n\in \mathbb N$ that $x=0$. 
A vector subspace $I$ of $E$ is called an \emph{order ideal} or just an \emph{ideal} whenever $0\leq |x|\leq |y|$ and $y\in I$ imply $x\in I$. Given a non-empty subset $A$ consisting of positive vectors in $E$, there is the smallest ideal $I_S$ in $E$ containing the set $S$. It is a standard fact that 
$$I_S=\left\{y\in E\colon 0\leq |y|\leq \sum_{k=1}^n \lambda_k x_k \textrm{ for some } n\in\mathbb N,\ \lambda_1,\ldots,\lambda_n\geq 0, \textrm{ and } x_1,\ldots,x_n\in S\right\}.$$
The ideal $I_S$ is called the \emph{ideal generated by} $S$. If $S=\{x\}$, then $I_S$ is said to be \emph{principal} and we write $I_x$ instead of $I_{\{x\}}$. It should be clear that the ideal $I_S$ is principal if and only if $S$ is finite.  If $I_x=E$ for some positive vector $x$, then $x$ is called a \emph{strong unit}. 

A proper ideal $M$ of $E$ is called \emph{maximal} whenever for each ideal $J$ satisfying $M\subseteq J\subseteq E$ it follows that either $J=M$ or $J=E$. If the co-dimension of an ideal $I$ in $E$ is one, then $I$ is a maximal ideal and conversely, all maximal ideals have co-dimension one by \cite[Corollary
~p.66]{Schaefer}. If a vector lattice contains a strong unit, then every proper ideal is contained in a maximal ideal (see \cite[Theorem 27.4]{Zaanen}). If $E$ does not contain a strong unit, then it may happen that there are no maximal ideals in $E$ (see \cite[Example 27.8]{Zaanen}). A vector lattice contains a strong unit whenever it contains a principal maximal ideal. In fact, we prove a more general statement below.  

\begin{lemma}\label{l: strong units wrt principal quotients}
Let $I$ be a principal ideal in a vector lattice $E$ such that the quotient vector lattice has a strong unit. Then $E$ has a strong unit.  
\end{lemma}

\begin{proof}
 Suppose that there exists a positive vector $x\in E$ such that $I=I_x$ and let $y+I$ be a strong unit for $E/I$ for some positive vector $y\in E$. Pick a vector $u\in E$. Then there exists a $\lambda>0$ such that 
 $|u|+I\leq \lambda y+I$. From the definition of the order on a quotient vector lattice, it follows that there exists a positive vector $w\in I_x$ such that 
 $|u|\leq \lambda y+w$. Hence, there exists a $\mu>0$ such that $|u|\leq \lambda y+\mu x$ which proves that $x+y$ is a strong unit for $E$. 
\end{proof}

The set of all maximal ideals is completely characterized in the vector lattice $C(K)$ of continuous functions on a compact Hausdorff space $K$. It is well-known that an ideal $I$ in $C(K)$ is maximal if and only if $I$ is of the form 
$$M_x:=\{f\in C(K)\colon f(x)=0\}$$ for some $x\in K$. If $E$ is a sublattice of $C(K)$, then for $x\in K$ we denote  the ideal $\left\{f\in E\colon f(x)=0\right\}$ by $M_x^E$. The norm convergence in the Banach space $C(K)$ is sometimes referred to as uniform convergence. This notion can be generalized to vector lattices as follows. We say that a net $(x_\alpha)_{\alpha}$ in a vector lattice $E$ converges \emph{relatively uniformly} to some $x\in E$ if there exists a positive vector $y$, the regulator, such that for each $\epsilon>0$ there exists an index $\alpha_\epsilon$ such that for all $\alpha\geq \alpha_\epsilon$ we have 
$|x_\alpha-x|\leq \epsilon y$. In general, a sequence can converge relatively uniformly to more than one vector. If every relatively uniform Cauchy net converges relatively uniformly in $E$ with respect to the same regulator, then $E$ is called \emph{uniformly complete}. If $E$ Archimedean, then the relative uniform limit is uniquely determined.  In Archimedean vector lattices the general notion of relative uniform convergence is not far away from the classical one for $C(K)$ spaces. Given a positive vector $x$ in an Archimedean vector lattice $E$, the principal ideal $I_x$ equipped with
$$\|y\|_x:=\inf\{\lambda>0:\; |y|\leq \lambda x\}$$
is a normed space. The norm completion of $(I_x,\|\cdot\|_x)$ is an AM-space with a strong unit $x$ which by the Kakutani representation theorem (see \cite[Theorem 4.29]{AliprantisBurkinshaw}) is lattice isometric to the Banach lattice $C(K)$ for some compact Hausdorff space $K$ where the strong unit $x$ is mapped to the constant one function. Hence, if the normed lattice $(I_x,\|\cdot\|_x)$ is norm complete, then $I_x$ itself is already lattice isometric to $C(K)$.  It is quite easy to see that $E$ is relatively uniformly complete if and only if $(I_x,\|\cdot\|_x)$ is a Banach lattice for each positive vector $x\in E$. 

An ideal $P\subsetneq E$ is said to be \emph{prime} whenever $x\wedge y\in P$ implies $x\in P$ or $y\in P$.  By \cite[Theorem~33.3]{Zaanen} every maximal ideal is prime. For the proof of the following useful characterization of prime ideals we refer the reader to \cite[Theorem 33.2]{Zaanen}. 

\begin{theorem}
For an ideal $P$ in a vector lattice $E$ the following conditions are equivalent. 
\begin{itemize}
    \item[$(i)$] $P$ is prime. 
    \item[$(ii)$] If $x\wedge y=0$, then $x\in P$ or $y\in P$. 
    \item[$(iii)$] The quotient vector lattice $E/P$ is linearly ordered. 
    \item[$(iv)$] For any ideals $I$ and $J$ satisfying $I\cap J\subseteq P$ we have $I\subseteq P$ or $J\subseteq P$. 
\end{itemize}
\end{theorem}

If $P$ is a prime ideal, then by \cite[Theorem 33.3]{Zaanen} every ideal containing $P$ is also prime, and furthermore, the family of all ideals containing $P$ is linearly ordered. 

A non-zero positive vector $a\in E$ is said to be an \emph{atom} whenever it follows from $0\leq x\leq a$, $0\leq y\leq a$  and $x\wedge y=0$ that $x=0$ or $y=0$. If the principal ideal $I_a$ generated by $a$ equals the linear span of $a$, then $a$ is a \emph{discrete element}. While every discrete element is an atom, both classes coincide in Archimedean vector lattices. Examples of atomic vector lattices are $\mathbb R^n$ ordered coordinatewise and the vector lattice $c_{00}(\Omega)$ of all functions with finite support ordered pointwise on a non-empty set $\Omega$. The atoms in $c_{00}(\Omega)$ are precisely the positive multiples of characteristic functions of singleton sets.
In an Archimedean vector lattice  the principal ideal $I_a$ generated by an atom $a$ is a projection band and so $M_a:=\{a\}^d$ is a maximal ideal in $E$.  
The linear span $A_0$ of all atoms in a vector lattice $E$ is always an ideal. The \emph{atomic} part $A$ of $E$ is the band generated by $A_0$. The disjoint complement $C=A^d$ is called the \emph{continuous} or \emph{atomless} part of $E$. If $C=0$ or equivalently $A^{dd}=E$, then $E$ is an \emph{atomic} vector lattice. An ideal $I$ in $E$ with the property $I^d=\{0\}$ is said to be \emph{order dense}. It turns out that $I$ is order dense in $E$ if and only if for each non-zero positive vector $x\in E$ there is a non-zero positive vector $y\in I$ such that $0<y\leq x$. We refer the reader to the standard texts \cite{Zaanen}, \cite{Schaefer}, \cite{AliprantisBurkinshaw} and \cite{AbramovichAliprantis} for any  unexplained terminology about vector lattices.

\section{Order dense prime ideals in vector lattices}\label{section: order dense prime ideals}

It turns out that order dense prime ideals are closely related to the existence of atoms, in fact, their dichotomous connection is proved in \Cref{P: order dense prime ideals}. For example, if an Archimedean vector lattice $E$ contains an atom $a$, then its disjoint complement $M_a=\{a\}^d$ is of co-dimension one, so it is a maximal ideal in $E$. Since $a$ is disjoint with $M_a$, it follows that the prime ideal $M_a$ is not order dense in $E$. The canonical Archimedean vector lattice which does not contain any order dense prime ideals is of the form $c_{00}(\Omega)$ for some set $\Omega$ as stated in \Cref{P: order dense prime ideals}. The equivalent statements in the proposition below follow from a combination of \cite[p.~430]{Zaanen} and \cite[Theorem~61.4]{Zaanen}. 

\begin{proposition}\label{T:maximal and prime ideals in atomic}
Let $\Omega$ be a non-empty set. Then an ideal $I\subseteq c_{00}(\Omega)$ is maximal if and only if $I$ is of the form  
\[
M_w:=\{x\in c_{00}(\Omega)\colon x(w)=0\} 
\]
for some $w\in\Omega$. Moreover, for a proper ideal $I$ in $c_{00}(\Omega)$ the following statements are equivalent.
\begin{itemize}
    \item[$(i)$] $I$ is a minimal prime ideal.
    \item[$(ii)$] $I$ is a prime ideal.
    \item[$(iii)$] $I$ is a maximal ideal.
\end{itemize}
\end{proposition}

\begin{proof}
For $w\in\Omega$ the ideal $M_w$ is maximal since it has co-dimension one. If $I$ is maximal in $c_{00}(\Omega)$, then it follows that $I\subseteq M_w$ for some $w\in\Omega$. Indeed, if for every $w\in\Omega$ there is a function $f\in I$ such that $f(w)>0$, this would imply $I=c_{00}(\Omega)$. Hence, there must exist $w\in\Omega$ such that $f(w)=0$ for all $f\in I$. Since $I$ is a maximal ideal, we conclude that $I=M_w$. 
\end{proof}

The dichotomous connection between the existence of order dense prime ideals and the existence of atoms is made precise in the following theorem.

\begin{theorem}\label{P: order dense prime ideals}
Let $E$ be an Archimedean vector lattice. 
\begin{enumerate}
\item[$(i)$] $E$ is atomless if and only if every prime ideal in $E$ is order dense in $E$.   
\item[$(ii)$] None of the prime ideals are order dense in $E$ if and only if $E$ is atomic and $E=A_0$.
\end{enumerate}
\end{theorem}

\begin{proof}
$(i)$ Suppose that $E$ contains a prime ideal $P$ that is not order dense. Then $P^d\neq \{0\}$. 
If $\dim P^d\geq 2$, then as $E$ is Archimedean, there must exist two vectors $x_1$ and $x_2$ in $P^d$ that are incomparable. It follows that $x:=x_1-x_1\wedge x_2$ and $y:=x_2-x_1\wedge x_2$ are now two non-zero disjoint vectors in $P^d$. Since neither $x$ nor $y$ is in $P$, this contradicts the fact that it is a prime ideal. Hence $\dim P^d=1$, so that $P^{dd}$ is necessarily a maximal band in $E$. By \cite[Theorem 26.7]{Zaanen} we conclude that there exists an atom $a\in E$ such that $P^{d}=\{a\}^d$. 

To prove the converse implication, assume that $E$ contains an atom. Then $M_a$ is a maximal ideal in $E$. Since there is no positive vector $x\in M_a$ such that $0<x\leq a$, the prime ideal $M_a$ is not order dense in $E$. 

$(ii)$ Suppose that $E$ does not contain order dense prime ideals. We first show that $E$ has the projection property. To see this, pick a band $B$ in $E$ and consider the sum $J:=B\oplus B^d$. Then $J$ is an order dense ideal in $E$. If $J\neq E$, then there exists a vector $x\in E\setminus J$, so that by \cite[Theorem 33.4]{Zaanen} there exists a proper prime ideal $P$ containing $J$ which does not contain $x$. Since $P$ is necessarily order dense, this contradicts our assumption. It follows that $E=B\oplus B^d$ and so $B$ is a projection band. 

To prove that $E$ is atomic, note that by $(i)$ the vector lattice $E$ contains an atom.  We need to prove that $E=A$ or equivalently $C=\{0\}$.
Suppose that $C\neq \{0\}$. Then $C$ is a non-trivial atomless band in $E$. By $(i)$, there is an order dense prime ideal $P$ in $C$. The set $A\oplus P$ is clearly an ideal in $E$. Since 
$(A\oplus P)^d=A^d\cap P^d=C\cap P^d$ and since $P$ is order dense in $C$, we conclude that $A\oplus P$ is order dense in $E$. 

We claim that $A\oplus P$ is a prime ideal in $E$. To see this, pick vectors $x$ and $y\in E$ with $x\wedge y=0$. 
Since $E=A\oplus C$, $x$ and $y$ can be decomposed as $x=x_1+x_2$ and $y=y_1+y_2$ where $x_1,y_1\in A$ and $x_2,y_2\in C$. From
$$x\wedge y=x_1\wedge y_1 + x_2\wedge y_2=0$$ 
we conclude that $x_2\wedge y_2=0$. Since $P$ is a prime ideal in $C$, we conclude that $x_2\in P$ or $y_2\in P$. This proves that $x\in A\oplus P$ or $y\in A\oplus P$, and so $A\oplus P$ is an order dense prime ideal in $E$ which is impossible. Hence $C=\{0\}$, which proves that $E$ is atomic. If $A_0\neq E$, then there exists a proper prime ideal $P$ in $E$ which contains $A_0$. Since $E$ is atomic, it follows that $A_0$, and therefore also $P$, is order dense in $E$ which is impossible by assumption, hence $E=A_0$. 

To conclude the proof note that by \Cref{T:maximal and prime ideals in atomic} the class of prime ideals coincides with the class of maximal ideals in $A_0$ and that none of the maximal ideals in $A_0$ are  order dense. 
\end{proof}

\begin{remark}
Note that statement $(i)$ in \Cref{P: order dense prime ideals} can be reformulated so that it characterizes when non-order dense prime ideals exist in $E$. That is, an Archimedean vector lattice $E$ contains a prime ideal that is not order dense if and only if $E$ contains an atom.  
\end{remark}

\section{Vector lattices with finitely many prime ideals}
\label{section: finitely many prime ideals}

In this section we consider vector lattices which have finitely many (minimal) prime ideals and prove in \Cref{T: finitely many minimal primes} that these vector lattices must be finite-dimensional. We start with a general lemma that locates prime ideals and characterizes maximal ideals of norm dense sublattices of $C(K)$-spaces. 

\begin{lemma}\label{L:lemma about maximals}
Let $K$ be a compact Hausdorff space and let  $E$ be a sublattice of $C(K)$. 
\begin{itemize}
    \item [$(i)$] For every proper ideal $J$ in $E$ there exists a point $x\in K$ such that $J$ is contained in $M_x^E$.
    \item[$(ii)$] If $E$ is norm dense, then an ideal $J$ in $E$ is maximal if and only if it is of the form $M_x^E$ for a unique point $x\in K$.
    \item [$(iii)$] If $E$ is norm dense, then for every prime ideal $P$ in $E$ there exists a unique point $x\in K$ such that $P$ is contained in $M_x^E$. 
\end{itemize}
\end{lemma}

\begin{proof}
$(i)$ Suppose that for each $x\in K$ there exists $f_x\in J$ such that $f_x(x)>0$. By continuity, for each $f_x\in J$ one can find an open neighborhood $U_x$ of $x$ such that $f_x$ is strictly positive on $U_x$. Since we have an open covering $\{U_x\colon x\in K\}$ of $K$, there exists a finite subcover $U_{x_1},\ldots,U_{x_n}$. The function $f:=f_{x_1}+\cdots+f_{x_n}$ is strictly positive on $K$, so it is a strong unit in $E$ and therefore, we have $J=E$. This contradiction shows that there exists $x\in K$ such that $J\subseteq M_x^E$. 

$(ii)$ Let $x\in K$ and consider the ideal $M_x^E$ in $E$. It follows that $M_x^E$ is proper, as otherwise $f(x)=0$ for all $f\in C(K)$ since $E$ is norm dense in $C(K)$. Suppose that $J$ is a proper ideal in $E$ containing $M_x^E$. By $(i)$ there is a point $y\in K$ such that $J\subseteq M_y^E$. If there would exists a function $f\in M_y^E\setminus M_x^E$, then for $h\in E\setminus M_y^E$ we can define $g\in E$ by 
\[
g(z):=f(z)-\frac{f(x)}{h(x)}h(z). 
\]
But then $g(x)=0$ and $g(y)\neq 0$, which contradicts the fact that $M_x^E\subseteq M_y^E$. Hence $M_x^E=M_y^E$, and so $J=M_x^E$ showing that $M_x^E$ is a maximal ideal in $E$. 

On the other hand, if $J$ is a maximal ideal in $E$, then there is a point $x\in K$ such that $J\subseteq M_x^E$ by $(i)$ and so $J=M_x^E$. 

For the uniqueness of the point $x\in K$, suppose that $M_x^E=M_y^E$ for $x\neq y$. We claim that for any $f,g\in E$ it follows that $f(x)g(y)=f(y)g(x)$. Indeed, if $f(x)=0$, then $f\in M_x^E=M_y^E$, so $f(y)=0$ and $f(x)g(y)=f(y)g(x)$. If $f(x)\neq 0$, then $h\in E$ defined by 
\[
h(z):=g(z)-\frac{g(x)}{f(x)}f(z) 
\]
satisfies $h(x)=0$, so $h\in M_x^E=M_y^E$ and $h(y)=0$ implies that $f(x)g(y)=f(y)g(x)$, proving the claim. Since $E$ is norm dense in $C(K)$, it follows that $f(x)g(y)=f(y)g(x)$ for all $f,g\in C(K)$, but this contradicts Urysohn's lemma. We conclude that $x=y$ and the maximal ideals in $E$ are therefore uniquely determined by the points in $K$. 

$(iii)$ Let $P$ be any prime ideal in $E$. Then by $(i)$ there is a point $x\in K$ such that $P\subseteq M_x^E$. Suppose that there is another point $y\in K$ such that $P\subseteq M_y^E$. Since $P$ is a prime ideal, it follows that we must have $M_x^E\subseteq M_y^E$ or $M_y^E\subseteq M_x^E$ and so $M_x^E=M_y^E$. Hence, $x=y$ by $(ii)$. 
\end{proof}

Although \Cref{L:lemma about maximals} is stated for sublattices of $C(K)$-spaces, it is useful in a more
general setting when studying prime ideals in Archimedean vector lattices. Indeed, if $E$ is an Archimedean vector lattice and $x\in E$ is a non-zero positive vector, then by the Kakutani representation theorem, the normed lattice $(I_x,\|\cdot\|_x)$ is lattice isometric to a dense sublattice of $C(K)$ for some compact Hausdorff space $K$. Since any prime ideal $P$ in $E$ yields a prime ideal $P\cap I_x$ in $I_x$, this allows us to study the properties of the principal ideals in $E$ and the topological properties of $K$ given the restrictions on the set of prime ideals in question. 

\begin{lemma}\label{L:C(K) finite minimal primes}
Let $K$ be a compact Hausdorff space. If there is a norm dense sublattice $E$ of $C(K)$ that contains only finitely many minimal prime ideals, then $K$ is finite and $E=C(K)$.
\end{lemma}

\begin{proof}
Suppose $K$ contains a sequence of distinct points $(x_n)_{n\in\mathbb{N}}$ all of which correspond to maximal ideals $M_{x_n}^E$ by Lemma~\ref{L:lemma about maximals}. Since every maximal ideal contains a minimal prime ideal by \cite[Theorem~33.7]{Zaanen} and distinct maximal ideals do not contain the same minimal prime ideal by \cite[Theorem~33.3]{Zaanen}, it follows that $E$ must contain infinitely many minimal prime ideals, contradicting the assumption. Hence $K$ is finite and so $E=C(K)$ as it is norm dense. 
\end{proof}

The main result of this section is proved by using \Cref{L:lemma about maximals} since Archimedean vector lattices are saturated with copies of dense sublattices of $C(K)$-spaces.

\begin{theorem}\label{T: finitely many minimal primes}
If an Archimedean vector lattice $E$ contains finitely many minimal prime ideals, then it is finite-dimensional. 
\end{theorem}

\begin{proof}
Pick any positive vector $x\in E$ and consider the principal ideal $I_x$. We claim that $I_x$ contains only finitely many minimal prime ideals. Suppose that $I_x$ contains infinitely many minimal prime ideals $\{P_n\colon n\in \mathbb {N}\}$. For each $n\in \mathbb N$ there exists a minimal prime ideal $Q_n$ in $E$ such that $P_n=Q_n\cap I_x$ by \cite[Theorem~52.3]{Zaanen}. Since $P_n\neq P_m$ for $n\neq m$, the family $\{Q_n\colon n\in\mathbb {N}\}$ of minimal prime ideals of $E$ is infinite. Thus, $I_x$ contains only finitely many minimal prime ideals. 

The Banach lattice completion of the normed space $(I_x,\|\cdot\|_x)$ is lattice isometric to $C(K)$ for some compact Hausdorff space $K$ by the Kakutani representation theorem. Hence, $C(K)$ contains a norm dense sublattice with only finitely many minimal prime ideals. By Lemma~\ref{L:C(K) finite minimal primes} we conclude that $K$ is finite, so that $I_x$ is finite-dimensional. 
%Hence the principal ideals of $E$ are all projection bands, so that an application of 
Hence, by \cite[Theorem~61.4]{Zaanen} we have that $E$ is lattice isomorphic to the space $c_{00}(\Omega)$ for some set $\Omega$. By Proposition~\ref{T:maximal and prime ideals in atomic} we have that $\Omega$ is finite, making $E$  finite-dimensional. 
\end{proof}

\section{Cohen's and Kaplansky's theorem in vector lattices}
\label{section: Cohen}
In this section we prove vector lattice analogues of two well-known results in the theory of commutative rings that involve prime ideals. Cohen's theorem (see \cite[Theorem~2]{Cohen50}) states that a commutative ring is Noetherian if and only if every prime ideal is finitely generated, and Kaplansky's theorem (see \cite[Theorem 12.3]{Kaplansky49}) states that in a Noetherian commutative ring every ideal is principal if and only if every maximal ideal is principal. Combining Cohen's theorem and Kaplansky's theorem yields that a commutative ring is a principal ideal ring if and only if every prime ideal is principal, which is referred to as the Cohen-Kaplansky theorem. In vector lattices an ideal is finitely generated if and only if it is principal, so Cohen's theorem becomes a statement about principal prime ideals in vector lattices. We prove the vector lattice analogue of the Cohen-Kaplansky theorem stating that every prime ideal in an Archimedean vector lattice is principal if and only if it is finite-dimensional, which is further equivalent with all ideals being principal. Furthermore, we prove that in a uniformly complete Archimedean vector lattice there are no non-maximal principal prime ideals, and we use this result to prove a vector lattice analogue of the Cohen-Kaplansky theorem for uniformly complete Archimedean vector lattices. That is, every ideal is principal if and only if every maximal ideal is principal, which in turn is equivalent with the vector lattice being finite-dimensional. 

We say that a vector lattice $E$ is \emph{Noetherian} if every ascending chain of ideals in $E$ is stationary. This notion is completely analogous to that of Noetherian rings or modules. Since Noetherian vector spaces are precisely the finite-dimensional ones, the following result should not be too surprising. 

\begin{proposition}\label{P: Noetherian lattice}
For an Archimedean vector lattice $E$ the following statements are equivalent.
\begin{enumerate}
 \item[$(i)$] $E$ is finite-dimensional.
 \item[$(ii)$] $E$ is Noetherian. 
\end{enumerate}
\end{proposition}

\begin{proof}
Since every finite-dimensional vector lattice can only have finitely many ideals, it is clear that $(i)$ implies $(ii)$. Assume that $E$ is infinite-dimensional. By \cite[Theorem~ 26.10]{Zaanen} $E$ contains an infinite sequence $(e_n)_{n\in \N}$ of pairwise disjoint positive vectors.
If for each $n\in \mathbb N$ we define the ideal $J_n$ generated by the set $\{e_1,\ldots,e_n\}$,  
then the ascending chain $J_1\subseteq J_2\subseteq\ldots$ of ideals in $E$ is not stationary.   
\end{proof}

In view of \Cref{P: Noetherian lattice} the vector lattice analogue of Cohen's theorem characterizes the finite-dimensional vector lattices among those for which all prime ideals are principal. Furthermore, a well-known alternative definition of a commutative ring being Noetherian is that every ideal is finitely generated, which leads to the following theorem. This result can be thought of as a combination of Cohen's theorem and the Cohen-Kaplansky theorem for vector lattices.

\begin{theorem}[Cohen-Kaplansky theorem for vector lattices]\label{T: Cohen}
Let $E$ be a vector lattice, and consider the following statements.
\begin{enumerate}
    \item[$(i)$] $E$ is finite-dimensional.
    \item[$(ii)$] Every proper ideal in $E$ is principal.
    \item[$(iii)$] Every prime ideal in $E$ is principal. 
\end{enumerate} 
Then $(i)$ implies $(ii)$, $(i)$ implies $(iii)$, and $(iii)$ implies $(ii)$. Moreover, in the case where $E$ is Archimedean, we have that $(ii)$ implies $(i)$, so all statements are equivalent. 
\end{theorem}

\begin{proof}
Since every finite-dimensional vector lattice has a strong unit, the fact that $(i)$ implies $(ii)$ and $(iii)$ is clear. We proceed to prove that $(iii)$ implies $(ii)$. To this end, suppose there is a proper ideal $I$ in $E$ which is not principal. Let $z\in E\setminus I$ be positive and define 
\[
\mathscr{I}:=\{J\subseteq E\colon J\ \mbox{is a non-principal ideal with $z\notin J$}\}.
\]
Then $\mathscr{I}$ is non-empty an can be partially ordered by set inclusion. For any chain $(J_i)_i$ in $\mathscr{I}$ it follows that its union $J_0:=\bigcup_i J_i$ is an ideal which is not principal. Indeed, if $J_0=I_x$ for some positive vector $x\in E$, then $x\in J_i$ for some $i$ and so $J_0\subseteq I_x\subseteq J_i\subseteq J_0$ hence $J_i=I_x$, which is impossible by definition of $\mathscr{I}$. Furthermore, $J_0$ does not contain $z$ making it a proper ideal in $E$. By Zorn's lemma $\mathscr{I}$ has a maximal element, say $P$. 

We claim that $P$ is a prime ideal in $E$. Suppose $x,y\in E$ are so that $x\wedge y=0$ and neither $x$ nor $y$ are in $P$. Then the prime ideals $P+I_x$ and $P+I_y$ are principal, so there are positive vectors $v,w\in E$ such that $P+I_x=I_v$ and $P+I_y=I_w$. It follows that $(P+I_x)\cap(P+I_y)=I_v\cap I_w=I_{v\wedge w}$ and by \cite[Proposition~II.2.3]{Schaefer} we find 
\[
(P+I_x)\cap(P+I_y)=P+P\cap I_x+P\cap I_y+I_x\cap I_y=P+I_{x\wedge y}=P.
\]
Hence, $P=I_{v\wedge w}$, but that is impossible as $P\in\mathscr{I}$. Thus, $P$ must be a prime ideal, but then by assumption it must be principal contradicting $P\in\mathscr{I}$ once more. We conclude that every proper ideal in $E$ must be principal.  

Let $E$ be Archimedean such that every proper ideal in $E$ is principal. To prove that $(ii)$ implies $(i)$ by contradiction, suppose that $E$ is not finite-dimensional. Then by \cite[Theorem 26.10]{Zaanen} there exists an infinite sequence $(x_n)_{n\in\N}$ of non-zero positive pairwise disjoint vectors in $E$.  Let $I$ be the ideal generated by $\{x_{n+1}\colon n\in\N\}$ in $E$. By construction $I$ is proper, so it is principal and there exists $x\in E$ such that $I=I_x$. Hence there are $x_{k_1},\ldots,x_{k_m}$ and positive scalars $\lambda_1,\ldots,\lambda_m$ such that 
$$0<x\le \lambda_1 x_{k_1}+\lambda_2 x_{k_2}+\cdots+\lambda_m x_{k_m}.$$ 
Thus there is an $n>1$ such that $x\perp x_{n}$. Since $x_{n}\in I_x$, we conclude that $x_{n}=0$ which is impossible. This contradiction shows that $E$ is finite-dimensional. 
\end{proof}

By \Cref{T: Cohen} it is clear that every infinite-dimensional Archimedean vector lattice contains a non-principal prime ideal. The following example shows that there even exists an atomic infinite-dimensional Archimedean vector lattice whose maximal ideals are all principal, yet none of the remaining prime ideals are principal. 

\begin{example}\label{E:atomic with only max principal}
Let $x:=(x_n)_{n\in\N}$ be an element of $c_0$ and suppose $x_n>0$ for all $n\ge 1$. By adjoining the constant sequence $\mathbf{1}$ to the ideal $I_x$ generated by $x$ in $c$ we obtain the atomic vector lattice $E:=I_x+\mathbb{R}\mathbf{1}$ that is not uniformly complete, and its uniform completion equals $c$.  Since $c$ is lattice isometric to $C(\mathbb{N}\cup\{\infty\})$, where $\mathbb{N}\cup\{\infty\}$ denotes the one-point compactification of $\mathbb{N}$, the maximal ideals in $E$ are of the form $M_n^E=M_n\cap E$ for a maximal ideal $M_n$ in $C(\mathbb{N}\cup\{\infty\})$ by Lemma~\ref{L:lemma about maximals}. For $n\in\mathbb{N}$ the maximal ideal $M_n^E$ is principal with generator $\mathbf{1}-e_n\in E$ and for the adjoined point it follows that $M_\infty=c_0$ and so $M_\infty^E=I_x$ which is principal by construction. 

We claim that no non-maximal prime ideal $P$ can be contained in $M_n^E$ for some $n\in\mathbb N$. Indeed, if there is an $n\in\mathbb{N}$ such that $P\subsetneq M_n^E$, then it follows from the fact that $e_n\notin M_n^E$ and $e_n\wedge (\mathbf{1}-e_n)=0$ that $\mathbf{1}-e_n\in P$, but this would imply that $P=M_n^E$. Thus, there are no proper non-maximal prime ideals contained in $M_n^E$ for all $n\in\mathbb{N}$. 

Suppose now that there exists a principal prime ideal $P$ which is not maximal. Since $P$ is not a subset of $M_n^E$ for $n\in\mathbb N$, we have $P\subseteq I_x$. If $y\in I_x$ is positive such that $P=I_y$, then $y_n>0$ for all $n\ge 1$. Moreover, for all $k\in\mathbb{N}$ there are $x_{n_k}$ and $y_{n_k}$ such that $x_{n_k}>ky_{n_k}$ since $I_y$ is properly contained in $I_x$. Define $v:=(v_m)_{m\in\N}$ and $w:=(w_m)_{m\in\N}$ as follows. 
If $m=n_{2k}$ for some $k\in \mathbb N$, then put $v_m=x_{n_{2k}}$ otherwise put $v_m=0$. Similarly, if $m=n_{2k-1}$ for some $k\in \mathbb N$, then put $w_m=x_{n_{2k-1}}$ otherwise put $w_m=0$. Then $v,w\in I_x$ and $v\wedge w=0$, but there is no multiple of $y$ that dominates $v$ or $w$. We conclude that $P$ is not principal and  $E$ is an infinite-dimensional atomic Archimedean vector lattice that is not uniformly complete in which all maximal ideals are principal and all non-maximal prime ideals are not principal. Note again that by \Cref{T: Cohen} such non-maximal prime ideals must exist in $E$.
\end{example}

Hence, by \Cref{E:atomic with only max principal} it is not true in general that an Archimedean vector lattice in which all maximal ideals are principal must be finite-dimensional. However, when passing to uniformly complete Archimedean vector lattices, it follows that the Cohen-Kaplansky theorem reduces to the vector lattice analogue involving only principal maximal ideals. 

\begin{theorem}\label{Theorem: principal prime is maximal}
Let $E$ be an Archimedean uniformly complete vector lattice. 
\begin{itemize}
    \item[$(i)$] If $E$ has a strong unit, then every principal prime ideal is maximal and equals the disjoint complement of an atom.
    \item[$(ii)$] If $E$ contains a principal prime ideal, then $E$ has a strong unit.
\end{itemize}
\end{theorem}

\begin{proof}
$(i)$ Suppose that $E$ has a strong unit $e$. Since $E$ is uniformly complete it is lattice isometric to $C(K)$ for some compact Hausdorff space $K$ by the Kakutani representation theorem. Hence, we may assume without loss of generality that $E=C(K)$. Let $P=I_f$ be a principal prime ideal in $E$ for some non-negative function $f\in E$. By \cite[Theorem 34.1]{Zaanen} there exists a unique point $x_0$ such that each function in $P$ vanishes at $x_0$. Since $P$ is principal, the point $x_0$ is the unique zero of $f$. 

We will show that $x_0$ must be an isolated point in $K$. Indeed, suppose that
$x_0$ is not isolated in $K$. Then we can inductively construct a sequence $(x_n)_{n\in\N}$ in $K$ such that $(f(x_n))_{n\in\N}$ is a strictly decreasing sequence with $f(x_n)\to 0$. For $n\ge 1$ define the open neighbourhoods of $x_0$ by
\[
U_n:=\{x\in K\colon f(x)<f(x_n)\}.
\]
As $x_{n+1}\in U_n\setminus U_{n+1}$, it follows that $U_{n+1}\subsetneq U_n$ for all $k\ge 1$, and since $\overline{U}_n\subseteq\{x\in K\colon f(x)\le f(x_n)\}$, it follows that $\overline{U}_{n+1}\subseteq U_n$ for all $n\ge 1$. Furthermore, the closure $\overline{U}_{n+2}$ is properly contained in $U_n$ for all $n\ge 1$. Hence, the open neighborhoods $V_n:=U_{2n-1}$ of $x_0$ satisfy $\overline{V}_{n+1}\subsetneq V_n$ for $n\ge 1$ and we define the non-empty closed sets $F_n:=\overline{V}_{2n-1}\setminus V_{2n}$ for $n\ge 1$. Note that for $n<m$ we have
\[
F_n\cap F_m=\overline{V}_{2n-1}\cap \overline{V}_{2m-1} \cap V_{2n}^c \cap  V_{2m}^c=\overline{V}_{2m-1} \cap V_{2n}^c \subseteq V_{2n} \cap V_{2n}^c=\emptyset,
\]
so the closed sets $F_n$ are pairwise disjoint. Define for $n\ge 2$ the continuous function $f_n$ on the closed set $F_n\cup\overline{V}_{2n+1}\cup V_{2n-2}^c$ by
\[f_n(x):=
\begin{cases}
\sqrt{f(x)}& \mbox{if } x\in F_n\\ \hskip 15pt 0 & \mbox{if } x\in \overline{V}_{2n+1}\cup V_{2n-2}^c,
\end{cases}
\]
which by the Tietze extension theorem can be extended to a positive continuous function on $K$ that will also be denoted by $f_n$ such that $f_n\le \sqrt{f}$. The functions $f_{2n}$ and $f_{2m}$ are disjoint for all $m\neq n$ and we define the sequences of partial sums $g_n:=\sum_{k=1}^n f_{4k}$ and $h_n:=\sum_{k=1}^n f_{4k-2}$. Note that for $n$ and $m$ the functions $g_n$ and $h_m$ are disjoint and that for $m<n$ we have $(g_n-g_m)(x)\le \sqrt{f(x)}$ on $V_{8m+6}$ and $(g_n-g_m)(x)=0$ on $V^c_{8m+6}$, so that $(g_n)_{n\in\N}$ is a Cauchy sequence in $C(K)$ for $\|\cdot\|_\infty$. Similarly, we find that $(h_n)_{n\in\N}$ is a Cauchy sequence in $C(K)$ for $\|\cdot\|_\infty$. Let $g$ and $h$ be the uniform limits of $(g_n)_{n\in\N}$ and $(h_n)_{n\in\N}$, respectively, in $C(K)$. Since $g_n\wedge h_n=0$ for all $n\ge 1$, it follows that $g\wedge h=0$ and so either $g\in P$ or $h\in P$ as $P$ is a prime ideal. Suppose that $g\in P$. Then there is a $\lambda>0$ such that $f_{4n}\le g_n\le g\le \lambda f$ for all $n\ge 1$. Hence, for all $x\in F_{4n}$ we have that $\sqrt{f(x)}\le \lambda f(x)$ and since $x_{16n-1}\in F_{4n}$, it follows that $1\le\lambda\sqrt{f(x_{16n-1})}$ for all $n\ge 1$ which contradicts the fact that $f(x_n)\to 0$. Hence, $g\notin P$ and similarly we have $h\notin P$, so we conclude that $x_0$ must be an isolated point. 

The set $K\setminus \{x_0\}$ is therefore closed in $K$ and since $f$ is strictly positive on $K\setminus\{x_0\}$, it attains a strictly positive minimum on $K\setminus \{x_0\}$. Thus, for each function $g\in M_{x_0}$ there exists a $\lambda>0$ such that $g\leq \lambda f$ which implies that $P=I_f=M_{x_0}$ is a maximal ideal. Furthermore, the characteristic function $\chi_{x_0}$ is an atom in $E$ and $M_{x_0}=\{\chi_{x_0}\}^d$.

$(ii)$ Let $P=I_x$ be a prime ideal in $E$, and let $y\in E\setminus P$ be positive. Then $Q:=P+I_y=I_{x+y}$ is an ideal in $E$ and suppose that $Q\subsetneq E$. Then for a positive vector $z\in E\setminus Q$ it follows from the fact that ideals in uniformly complete vector lattices are uniformly complete that $E':=I_{x+y+z}$ is a uniformly complete vector lattice with a strong unit that contains a principal prime ideal that is not maximal. This contradicts $(i)$, so $Q=E$ and we conclude that $E$ has $x+y$ as a strong unit.  
\end{proof}

It will be shown in \Cref{S: Wild Examples} that there are vector lattices containing non-maximal principal prime ideals. In view of \Cref{Theorem: principal prime is maximal}, these vector lattices cannot be uniformly complete. 

\begin{corollary}\label{C:lemma about principal max ideals in C(K)}
Let $K$ be a compact Hausdorff space. For $x\in K$ the maximal ideal $M_x$ in $C(K)$ is a principal ideal if and only if $x$ is an isolated point of $K$.
\end{corollary}

\begin{proof}
It was shown in the proof of \Cref{Theorem: principal prime is maximal} that $x$ must be an isolated point of $K$ whenever $M_x$ is principal. On the other hand, if $x$ is an isolated point of $K$, then the characteristic function $\chi_x$ is continuous on $K$ and it follows that $M_x=I_f$ for $f:=\mathbf{1}-\chi_x$.  
\end{proof}

The vector lattice analogue of the Cohen-Kaplansky theorem for uniformly complete Archimedean vector lattices is the following.    

\begin{theorem}[Cohen-Kaplansky theorem for uniformly complete vector lattices]\label{Main Cohen}
The following statements are equivalent for a uniformly complete Archimedean vector lattice $E$. 
\begin{enumerate}
    \item[$(i)$] $E$ is finite-dimensional.
    \item[$(ii)$] Every proper ideal in $E$ is principal.
    \item[$(iii)$] $E$ contains maximal ideals, and every maximal ideal in $E$ is principal.
\end{enumerate} 
\end{theorem}

\begin{proof}
Statements $(i)$ and $(ii)$ are equivalent by \Cref{T: Cohen}, and the fact that $(i)$ implies $(iii)$ is clear. Suppose $E$ contains maximal ideals and all of them are principal. By \Cref{l: strong units wrt principal quotients} the vector lattice $E$ has a strong unit, so it is lattice isometric to a $C(K)$-space. It follows from \Cref{C:lemma about principal max ideals in C(K)} that every point in $K$ is isolated. We conclude that $K$ must be finite in order to remain compact, so $E$ must be finite-dimensional.
\end{proof}

\section{Prime Noetherian vector lattices}
As was proved in \Cref{P: Noetherian lattice}, all Noetherian Archimedean vector lattices are finite-dimensional. When considering chains of ideals in vector lattices, the prime ideals are also a natural class of ideals to study. For example, all ideals that contain a fixed prime ideal are prime ideals and this set is linearly ordered by set inclusion. The vector lattices for which the ascending chains of prime ideals are finite will be studied in this section, and we propose the following definition.
A vector lattice $E$ is said to be \emph{prime Noetherian} if every ascending chain of prime ideals $P_1\subseteq P_2\subseteq \dots$ in $E$ is stationary. Uniformly complete Archimedean prime Noetherian vector lattices are completely characterized in \Cref{P:C(K) prime Noetherian finte K} and \Cref{T:prime Noetherian uniformly complete}, and these characterizations depend on the existence of a strong unit. More specifically, if $E$ has a strong unit, then $E$ is finite-dimensional and in general $E$ must be lattice isomorphic to $c_{00}(\Omega)$ for some set $\Omega$.     
%Similar to the theory of commutative rings, there is a notion of a Krull dimension for prime Noetherian vector lattices. 
We start this section by studying how the prime Noetherian property transfers between sublattices, ideals, and the whole vector lattice.

\begin{proposition}\label{P: sublattice prime Noetherian}
Let $E$ be a vector lattice and let $F$ be a vector sublattice. If $E$ is prime Noetherian, then $F$ is prime Noetherian.
\end{proposition}

\begin{proof}
Suppose that $Q_1\subseteq Q_2\subseteq\ldots$ is an ascending chain of distinct proper prime ideals in $F$. We can now create an ascending chain of distinct prime ideals $P_1\subseteq P_2\subseteq\ldots$ in $E$ such that $Q_n=P_n\cap F$ for all $n\in \mathbb{N}$ by induction. Indeed, for $Q_1\subseteq Q_2$ there are distinct prime ideals $P_1\subseteq P_2$ in $E$ such that $Q_1=P_1\cap F$ and $Q_2=P_2\cap F$ by \cite[Theorem~52.4]{Zaanen}. Suppose that for $Q_1\subseteq Q_2\subseteq\ldots\subseteq Q_k$ there are distinct prime ideals $P_1\subseteq P_2\subseteq\ldots\subseteq P_k$ in $E$ such that $Q_i=P_i\cap F$ for $i=1,\ldots,k$. Consider the  canonical Riesz homomorphism $\pi\colon E\to E/P_k$ and note that $\pi(F)$ is a sublattice of $\pi(E)$. The zero ideal is prime in both quotient vector lattices, so that $\pi(Q_{k+1})$ is a prime ideal in $\pi(F)$ by \cite[Theorem~33.3(iii)]{Zaanen}. Furthermore, by \cite[Theorem~52.2]{Zaanen} there is a prime ideal $P'$ in $\pi(E)$ such that $\pi(Q_{k+1})=P'\cap \pi(F)$, and define $P_{k+1}:=\pi^{-1}(P')$. It is readily verified that $P_{k+1}$ is a prime ideal in $E$ that contains $P_{k}$. Furthermore, if $x\in P_{k+1}\cap F$ then there is a $p\in P_k$ such that $x-p\in Q_{k+1}$, hence $p\in P_k\cap F=Q_k$ so that $x\in Q_{k+1}$. Since for $x\in Q_{k+1}$ it follows that $\pi(x)\in P'$ we also have $x\in P_{k+1}$, so $Q_{k+1}=P_{k+1}\cap F$. Note that since $Q_{k}$ and $Q_{k+1}$ are distinct, we must have that $P_{k}$ and $P_{k+1}$ are distinct, which concludes the induction argument. Since we assumed $E$ to be prime Noetherian, the chain $Q_1\subseteq Q_2\subseteq\ldots$ must be stationary. Hence, $F$ is prime Noetherian. 
\end{proof}

If on the other hand we have a vector lattice $E$ with a prime Noetherian sublattice $F$, then it is not true in general that $E$ is prime Noetherian even if $F$ is an order dense ideal. See the paragraph preceding \Cref{C:prime Noetherian char uc}. In the case where $F$ is a projection band we can prove the following.  

\begin{proposition}
Let $B$ be a projection band in a vector lattice $E$. Then $E$ is prime Noetherian if and only if $B$ and $B^d$ are prime Noetherian.  
\end{proposition}

\begin{proof}
If $E$ is prime Noetherian, then $B$ and $B^d$ are prime Noetherian by Proposition~\ref{P: sublattice prime Noetherian}. 
Assume now that $B$ and $B^d$ are prime Noetherian and let $P_1\subseteq P_2\subseteq \ldots$ be a chain of prime ideals in $E$. Then 
$$P_1\cap B\subseteq P_2\cap B\subseteq \ldots \qquad \textrm{and} \qquad P_1\cap B^d\subseteq P_2\cap B^d\subseteq \ldots$$ are increasing chains of prime ideals in $B$ and $B^d$, respectively. Hence, there exists $n\in \mathbb N$ such that for all $m\geq n$ we have
$$P_m\cap B=P_{n}\cap B \qquad \textrm{and} \qquad P_m\cap B^d=P_{n}\cap B^d.$$
Since the lattice of ideals in a vector lattice is distributive, for each $m\geq n$ we obtain 
\[P_m=(P_m\cap B)\oplus (P_m\cap B^d)=(P_{n}\cap B)\oplus (P_{n}\cap B^d)=P_{n}.\qedhere\]
\end{proof}

In a prime Noetherian vector lattice there are large prime ideals in the sense that they have finite co-dimension and every prime ideal is always contained in a large prime ideal. 

\begin{proposition}\label{P:Noetherian fin dim quotients}
In a prime Noetherian vector lattice every prime ideal is contained in a prime ideal of finite co-dimension.
\end{proposition}

\begin{proof}
Pick a prime ideal $P$ in a prime Noetherian vector lattice $E$. Suppose that for each ideal $Q$ containing $P$ the vector lattice $E/Q$ is infinite-dimensional. Let $P_1:=P$. Let $P\subsetneq P_2$ be any proper ideal in $E$. By assumption $E/P_2$ is infinite-dimensional. Pick any proper ideal $P_2\subsetneq P_3$ in $E$. By assumption the vector lattice $E/P_3$ is infinite-dimensional. 
Inductively we can construct an ascending chain $P_1\subsetneq P_2\subsetneq P_3\subsetneq \dots$  of distinct prime ideals which contradicts the assumption that $E$ is prime Noetherian.   
\end{proof}

It is not true in general that vector lattices contain maximal ideals, however, prime Noetherian vector lattices always do.

\begin{proposition}
Let $E$ be an at least two-dimensional prime Noetherian vector lattice. Then every proper ideal of $E$ is contained in a maximal ideal.
\end{proposition}

\begin{proof}
Let $I$ be a proper ideal in $E$. By \cite[Theorem~33.5]{Zaanen} the ideal $I$ is contained in a non-trivial prime ideal $P$. By \Cref{P:Noetherian fin dim quotients} the ideal $P$ is contained in a prime ideal $Q$ such that the dimension of $E/Q$ is finite. Let
\[
Q=Q_1\subseteq Q_2 \subseteq \ldots \subseteq Q_{n-1} \subseteq Q_n=E
\]
be the maximal chain $\mathcal C$ of ideals between $Q$ and $E$. Clearly, $Q_{n-1}$ is then a maximal ideal in $E$ which contains $I$. 
\end{proof}

%\begin{proof}
%Let $I$ be a proper ideal in $E$. By \cite[Theorem~33.5]{Zaanen} the ideal $I$ is contained in a non-trivial prime ideal $P$. The family of ideals that contain $P$ is linearly ordered and finite as $E$ is prime Noetherian. It follows that the maximal element of this chain is a maximal ideal in $E$. 
%\end{proof} 

The following proposition will be useful for proving the main results of this section. It characterizes the maximal ideals in a uniformly complete vector lattice with a strong unit among the prime ideals with finite co-dimension.

\begin{proposition}\label{uniformcomplete p=m}
Let $E$ be an Archimedean uniformly complete vector lattice with a strong unit. Then a prime ideal $P$ is a maximal ideal if and only if $\dim E/P<\infty$.
\end{proposition}

\begin{proof}
If $\dim E=1$, there is nothing to prove. So we may assume that the dimension of $E$ is at least two. Furthermore, since $E$ is uniformly complete with a strong unit, by the Kakutani representation theorem we may assume that $E=C(K)$ for some compact Hausdorff space $K$ where $K$ contains at least two points. 

Pick a prime ideal $P$ in $C(K)$ such that $\dim C(K)/P<\infty$. Then there exists an $x\in K$ such that $P$ is contained in the maximal ideal $M_x$. Let
\[
P=P_1\subseteq P_2 \subseteq \ldots \subseteq P_{n-1} \subseteq P_n=M_x
\]
be the maximal chain $\mathcal C$ of ideals, that are necessarily prime, in $E$ between $P$ and $M_x$. Since $\mathcal C$ is maximal, the ideal $P_{n-1}$ is a maximal order ideal in $P_n$.
 
We claim that $P_{n-1}$ is a maximal algebra ideal in $P_n$.  To see that $P_{n-1}$ is also an algebra ideal in $M_x$, pick $f\in P_{n-1}$ and $g\in M_x$ and note that the inequality 
$|fg|\leq \|g\|_\infty\,|f|$ together with the fact that $P_{n-1}$ is an order ideal in $M_x$ yields that $fg\in P_{n-1}$. Since $P_{n-1}$ is a maximal order ideal in $P_n$, the co-dimension of $P_{n-1}$ in $P_n$ is one, so that $P_{n-1}$ is also maximal as an algebra ideal in $P_n$. 

We claim that there exists $y\in K\setminus\{x\}$ such that 
$$P_{n-1}=\{f\in C(K):\; f(x)=f(y)=0\}.$$
To see this, consider the locally compact Hausdorff space $K\setminus\{x\}$. It is a standard fact from general topology that the one-point compactification of $K\setminus\{x\}$ is homeomorphic to $K$ and that the mapping $\Phi\colon C_0(K\setminus\{x\})\to M_x\subseteq C(K)$ defined by 
\[
\Phi(f)(t):=
\begin{cases}
f(t) &\mbox{if}\ t\neq x,\\
\hskip .2cm 0 & \mbox{if}\ t=x,
\end{cases}
\]
is an isometric lattice and algebra isomorphism. Since $P_{n-1}$ is a maximal algebra ideal in $M_x$, it follows that $\Phi^{-1}(P_{n-1})$ is a maximal algebra ideal in $C_0(K\setminus\{x\})$. By \cite[Theorem~1.4.6]{Kaniuth} there exists $y\in K\setminus \{x\}$ such that 
\[
\Phi^{-1}(P_{n-1})=\{f\in C_0(K\setminus\{x\})\colon f(y)=0\},
\]
so that 
\[
P_{n-1}=\Phi(\Phi^{-1}(P_{n-1}))=\{\Phi(f)\colon f\in \Phi^{-1}(P_{n-1})\}=\{f\in C(K)\colon f(x)=f(y)=0\}.
\]
We claim that $P_{n-1}$ is not a prime ideal. To this end, note that Urysohn's lemma yields functions $f$ and $g$  in $C(K)$ such that $f(x)=g(y)=1$ and $f(y)=g(x)=0$. Then $f\wedge g$ belongs to the ideal $P_{n-1}$, yet neither $f$ nor $g$ belongs to $P_{n-1}$. This contradiction shows that $n=1$ so that $P=M_x$ is a maximal ideal. Since maximal ideals have co-dimension one, this concludes the proof.
\end{proof}

\begin{corollary}\label{C: prime noetherian p=m}
A uniformly complete Archimedean vector lattice $E$ with a strong unit is prime Noetherian if and only if every prime ideal in $E$ is maximal.  
\end{corollary}

\begin{proof}
If every prime ideal in $E$ is maximal, then $E$ is necessarily prime Noetherian. Suppose now that $E$ is a prime Noetherian vector lattice and pick a prime ideal $P$ in $E$. We will prove that $E/P$ is finite-dimensional since \Cref{uniformcomplete p=m} will yield then that $P$ is a maximal ideal in $E$. 

By way of contradiction, assume that $E/P$ is infinite-dimensional and let $P_1:=P$. Pick any non-maximal ideal $Q$ in $E$ that properly contains $P_1$. By \Cref{uniformcomplete p=m} we have that $E/Q$ is infinite-dimensional. Now let $P_2:=Q$. Inductively we can construct an ascending chain $P_1\subsetneq P_2\subsetneq \cdots$ of prime ideals in $E$ such that for each $n\in \mathbb N$ the dimension of $E/P_n$ is infinite. However, this contradicts the fact that $E$ is prime Noetherian. 
\end{proof}

The following proposition proves that the finite-dimensional vector lattices are precisely the uniformly complete Archimedean prime Noetherian vector lattices with a strong unit.

\begin{proposition}\label{P:C(K) prime Noetherian finte K}
Let $E$ be a uniformly complete Archimedean vector lattice with a strong unit. Then $E$ is prime Noetherian if and only if $E$ is finite-dimensional. 
\end{proposition}

\begin{proof}
By \Cref{C: prime noetherian p=m} every prime ideal in $E$ is a maximal ideal, so that by \cite[Theorem 37.6]{Zaanen} the quotient vector lattice $E/J$ is Archimedean for every ideal $J$ in $E$. Since $E$ is uniformly complete, \cite[Theorem 61.4]{Zaanen} yields that $E$ is lattice isomorphic to the vector lattice $c_{00}(\Omega)$ for some set $\Omega$. Since $E$ has a strong unit, $\Omega$ needs to be finite. Hence, the vector lattice $E$ is finite-dimensional. 

The converse follows from the fact that finite-dimensional vector lattices have only finitely many ideals.
\end{proof}

In general, the uniformly complete prime Noetherian vector lattices are characterized as $c_{00}(\Omega)$ for some set $\Omega$.

\begin{theorem}\label{T:prime Noetherian uniformly complete}
A uniformly complete Archimedean vector lattice $E$ is prime Noetherian if and only if it is lattice isomorphic to $c_{00}(\Omega)$ for some set $\Omega$.
\end{theorem}

\begin{proof}
Suppose that $E$ is a prime Noetherian vector lattice.
Pick a positive vector $x\in E$ and consider the principal ideal $I_x$ which is a uniformly complete vector lattice with a strong unit. Since $I_x$ is prime Noetherian by Proposition~\ref{P: sublattice prime Noetherian}, it follows that $I_x$ is finite-dimensional by Proposition~\ref{P:C(K) prime Noetherian finte K}. Hence, $E$ is lattice isomorphic to $c_{00}(\Omega)$ for some set $\Omega$ by \cite[Theorem 61.4]{Zaanen}. For the converse implication, note that $c_{00}(\Omega)$ is prime Noetherian by \Cref{T:maximal and prime ideals in atomic}. 
\end{proof}

If $E$ is equipped with a completely metrizable locally solid topology, then the prime Noetherian property implies that $E$ is finite-dimensional. 

\begin{lemma}\label{L:c00 is finite dimensional}
If $c_{00}(\Omega)$ is lattice isomorphic to a completely metrizable locally solid vector lattice $E$, then $\Omega$ is finite.
\end{lemma}

\begin{proof}
Let $\Phi\colon c_{00}(\Omega) \to (E,\tau)$ be a lattice isomorphism where $\tau$ is a completely metrizable locally solid topology on $E$. Then $\Phi^{-1}$ induces a completely metrizable locally solid topology $\tau'$ on $c_{00}(\Omega)$. Hence, $\Phi$ is an isomorphism between completely metrizable locally solid vector lattices. 

Suppose that $\Omega$ is infinite and let $(\omega_n)_{n\in\mathbb N}$ be a sequence in $\Omega$ of distinct points. Denote by $x_n$ the vector $\Phi(\chi_{\omega_n})$ in $E$. Since $\tau$ is metrizable, there exists a local basis $\{U_n\colon n\in \mathbb N\}$ of solid neighborhoods of zero in $E$  with the property that $U_{n+1}+U_{n+1}\subseteq U_n$ for each $n\in\mathbb N$. Since each set $U_n$ is absorbing, there exists a $\lambda_n>0$ such that $\lambda_n x_n\in U_n$. Denote the vector $\lambda_1x_1+\cdots+\lambda_nx_n$ by $s_n$. 

We claim that the sequence $(s_n)_{n\in\mathbb N}$ is a Cauchy sequence in $(E,\tau)$. 
Pick any neighborhood $U$ of zero in $(E,\tau)$ and find $n_0\in\mathbb N$ such that
$U_{n_0}\subseteq U$. By \cite[Exercise 2.1.14]{Tourky07} for all $m'\geq m>n_0$ we have 
$$s_{m'}-s_m=\lambda_{m+1}x_{m+1}+\cdots+\lambda_{m'}x_{m'}\in U_{m+1}+\cdots+U_{m'}\subseteq U_m\subseteq U_{n_0}\subseteq U$$
which proves the claim. Hence, the increasing sequence $(s_n)_{n\in\mathbb N}$ converges to some positive vector  $s\in E$. Since $\Phi$ is an isomorphism, the increasing sequence $(t_n)_{n\in\mathbb N}$ where $t_n=\lambda_1\chi_{\omega_1}+\cdots+\lambda_n\chi_{\omega_n}$ is convergent to the vector $\Phi^{-1}(s)$. However, the inequality $\Phi^{-1}(s)\geq \lambda_n\chi_{\omega_n}$ which holds for each $n\in\mathbb N$ yields that $\Phi^{-1}(s)$ has infinite support. This clearly contradicts the definition of the space $c_{00}(\Omega)$. Therefore, $\Omega$ is finite and the proof is completed.
\end{proof}

\begin{proposition}\label{C:Banach lattice prime Noetherian}
A completely metrizable locally solid vector lattice is prime Noetherian if and only if it is finite-dimensional. In particular, a Banach lattice is prime Noetherian if and only if it is finite-dimensional. 
\end{proposition}

\begin{proof}
Let $E$ be a completely metrizable locally solid vector lattice with the prime Noetherian property. Since $E$ is uniformly complete, it is lattice isomorphic to $c_{00}(\Omega)$ by Theorem~\ref{T:prime Noetherian uniformly complete}. By Lemma~\ref{L:c00 is finite dimensional} we conclude that $E$ is finite-dimensional. 
\end{proof}

By \Cref{C:Banach lattice prime Noetherian} we can construct an example of a vector lattice $E$ that contains a prime Noetherian sublattice as an order dense ideal, however $E$ itself is not prime Noetherian. Indeed, the vector lattice $c_0$ contains $c_{00}$ as an order dense ideal. 

\begin{corollary}\label{C:prime Noetherian char uc}
The following assertions are equivalent for a uniformly complete prime Noetherian Archimedean vector lattice $E$.
\begin{enumerate}
    \item[$(i)$] $E$ has a strong unit.
    \item[$(ii)$] $E$ is lattice isomorphic to a Banach lattice.
    \item[$(iii)$] $E$ is finite-dimensional. 
\end{enumerate}
\end{corollary}

The following example shows that there exists an infinite-dimensional prime Noetherian vector lattice with a strong unit which is not uniformly complete. 

\begin{example}
The vector lattice $E:=c_{00}+\mathbb{R}\mathbf{1}$ of all eventually constant sequences is not uniformly complete, and every quotient space of $E$ is Archimedean by \cite[Exercise~61.5]{Zaanen} and therefore, every proper prime ideal in $E$ is a maximal ideal by \cite[Theorem~37.6]{Zaanen}. Hence, $E$ is an infinite-dimensional atomic prime Noetherian vector lattice with a strong unit.
\end{example}

\section{Vector lattices of piecewise polynomials}\label{S: Wild Examples}

In this section the prime ideals in vector lattices of piecewise polynomials are studied. It turns out that this class of non-uniformly complete vector lattices is a source of insightful examples when studying principal prime ideals and prime Noetherian properties of vector lattices. 

For $n\in\mathbb{N}$ let $\mathrm{PPol}^n([a,b])$ be the vector lattice of piecewise polynomials of degree at most $n$ that are continuous on the interval $[a,b]$, and we shall denote the space of piecewise polynomials that are continuous on the interval $[a,b]$ without any bound on the degree by $\mathrm{PPol}([a,b])$. By the lattice version of the Stone-Weierstrass theorem all these spaces are uniformly dense in $C([a,b])$. For the system of (not necessarily open) neighborhoods $N(t_0)$ of $t_0\in [a,b]$, we define the ideal
\[
I_{N(t_0)}:=\left\{f\in E\colon f^{-1}(\{0\})\in N(t_0)\right\}.
\]
Note that for $t_0\in(a,b)$ the ideal $I_{N(t_0)}$ is not prime. To see this, consider the functions $f(t):=(t-t_0)_+$ and $g(t):=(t_0-t)_+$. Then $f\wedge g=0$, but neither $f$ nor $g$ are in $I_{N(t_0)}$. 
Furthermore, if we write $E$ for either $\mathrm{PPol}^n([a,b])$ or $\mathrm{PPol}([a,b])$, then it follows from Lemma~\ref{L:lemma about maximals} that all maximal ideals in $E$ are of the form $M_{t_0}^E$ for some $t_0\in[a,b]$. For $t_0\in (a,b]$ we define 
\[
L_{t_0}:=\left\{f\in M_{t_0}^E\colon \mbox{there exists a $\delta>0$ such that $f(t)=0$ for $t\in(t_0-\delta,t_0]$}\right\}
\]
and for $t_0\in[a,b)$ we define
\[
R_{t_0}:=\left\{f\in M_{t_0}^E\colon \mbox{there exists a $\delta>0$ such that $f(t)=0$ for $t\in[t_0,t_0+\delta)$}\right\}.
\]
In fact, these are exactly the minimal prime ideals in $E$.

\begin{lemma}\label{L:minimal prime in poly}
Let $E$ be either $\mathrm{PPol}^n([a,b])$ or $\mathrm{PPol}([a,b])$. The minimal prime ideals in $E$ are precisely $L_{t_0}$ for $t_0\in(a,b]$ and $R_{t_0}$ for $t_0\in[a,b)$.
\end{lemma}

\begin{proof}
It is straightforward to check that $L_{t_0}$ and $R_{t_0}$ are ideals. Suppose $f,g\in E$ are such that $f\wedge g=0$. Since $f$ and $g$ are piecewise polynomials and have only finitely many zeros when they are not constant, there must be a $\delta>0$ such that $f(t)=0$ on either $(t_0-\delta,t_0]$ or on $[t_0,t_0+\delta)$, or $g(t)=0$ on either $(t_0-\delta,t_0]$ or on $[t_0,t_0+\delta)$. Hence $f$ is in $L_{t_0}$ or in $R_{t_0}$, or $g$ is in $L_{t_0}$ or $R_{t_0}$, and if $f$ is not in $L_{t_0}$, then $g$ is in $L_{t_0}$. Similarly, if $g$ is not in $L_{t_0}$, then $f$ is in $L_{t_0}$. Thus $L_{t_0}$ is a prime ideal and it follows analogously that $R_{t_0}$ is a prime ideal. 

Next we will show that $I_{N(t_0)}$ is contained in every prime ideal in $M_{t_0}^E$. Indeed, if $P$ is a prime ideal in $M_{t_0}^E$ and $f\in I_{N(t_0)}$ is positive, then there is an $\epsilon>0$ such that $(t_0-\epsilon,t_0+\epsilon)$ is in $\{t\in[a,b]\colon f(t)=0\}$ and we can construct a piecewise linear continuous function $g$ that is zero outside $(t_0-\epsilon,t_0+\epsilon)$ and $g(t_0)=1$. Since $f\wedge g=0$, it follows that $f\in P$. 

We proceed to show that $L_{t_0}$ and $R_{t_0}$ are the minimal prime ideals in $E$. Suppose $P$ is a minimal prime ideal in $M_{t_0}^E$, and there is a function $f\in L_{t_0}$ that is not in $P$ and that there is a function $g\in R_{t_0}$ that is not in $P$. Then $f\wedge g$ is not in $P$, but $f\wedge g$ is in $I_{N(t_0)}$ contradicting the fact that $I_{N(t_0)}\subseteq P$. Hence, $L_{t_0}\subseteq P$ or $R_{t_0}\subseteq P$. Since $P$ is a minimal prime ideal, it follows that $P=L_{t_0}$ or $P=R_{t_0}$. Moreover, if $t_0\in(a,b]$, then the prime ideal $L_{t_0}$ contains a minimal prime ideal $P$. So, there is an $s\in[a,b]$ such that $P=L_s\subseteq L_{t_0}$ or $P=R_s\subseteq L_{t_0}$. Then by \Cref{L:lemma about maximals} it follows that $s=t_0$. If $R_{t_0}$ is contained in $L_{t_0}$, then we must have that $R_{t_0}=I_{N(t_0)}$, which is impossible. Hence, for all $t_0\in(a,b]$ the prime ideals $L_{t_0}$ are minimal. Similarly, for all $t_0\in[a,b)$ the prime ideals $R_{t_0}$ are minimal as well.
\end{proof}

For $t_0\in(a,b]$ any $f\in \mathrm{PPol}^n([a,b])$ has left derivatives at $t_0$, which we will denote by $f_-^{(j)}(t_0)$ for $0\le j\le n$ where it is understood that $f^{(0)}_-(t_0)=f(t_0)$. The left $j$-th derivatives of $f$ yield a map $\varphi_L\colon \mathrm{PPol}^n([a,b])\to \mathbb{R}^{n+1}$  defined by 
\begin{align}\label{formula Riesz hom left}
\phi_L(f):=(f(t_0),-f'_-(t_0),f_-^{''}(t_0),\ldots,(-1)^{n}f_-^{(n)}(t_0)).
\end{align}
Similarly for $t_0\in[a,b)$ and any $f\in \mathrm{PPol}^n([a,b])$ we have right derivatives $f_+^{(j)}(t_0)$ for $0\le j\le n$ which yield the map $\varphi_R\colon \mathrm{PPol}^n([a,b])\to \mathbb{R}^{n+1}$  defined by 
\begin{align}\label{formula Riesz hom right}
\phi_R(f):=(f(t_0),-f_+'(t_0),f_+^{''}(t_0),\ldots,(-1)^nf^{(n)}_+(t_0)), 
\end{align}
where we again put $f_+^{(0)}(t_0)=f(t_0)$. It follows that if we equip $\mathbb{R}^{n+1}$ with the lexicographical ordering, the maps $\varphi_L$ and $\varphi_R$ are Riesz homomorphisms.

\begin{lemma}\label{L: phi_L and phi_R are hom's}
Equip $\mathbb{R}^{n+1}$ with the lexicographical ordering. Then for $t_0\in(a,b]$ the map $\varphi_L$ as in \eqref{formula Riesz hom left} is a Riesz homomorphism with kernel $L_{t_0}$ and for $t_0\in[a,b)$ the map $\varphi_R$ as in \eqref{formula Riesz hom right} is a Riesz homomorphism with kernel $R_{t_0}$. 
\end{lemma}

\begin{proof}
If $f(t_0)>0$, then $|f|=f$ on a left neighborhood of $t_0$, and $\phi_L(f)>\phi_L(-f)$, so $\phi_L(|f|)=|\phi_L(f)|$. On the other hand, if $f(t_0)<0$, then $|f|=-f$ on a left neighborhood of $t_0$ and $\phi_L(f)<\phi_L(-f)$, so $\phi_L(|f|)=|\phi_L(f)|$. Suppose that 
\[
f(t_0)=f'_-(t_0)=\ldots=f_-^{(k)}(t_0)=0 
\]
for some $0\le k<n=\mathrm{deg}(f)$. Then on a left neighborhood of $t_0$ we have that $f$ is of the form 
\[
f(t)=a_{k+1}(t-t_0)^{k+1}+\ldots+a_{n-1}(t-t_0)^{n-1}+a_n(t-t_0)^n.
\]
If $f^{(k+1)}_-(t_0)=(k+1)!a_{k+1}>0$, then for $k$ odd we have that
\begin{align}\label{equation for principal 1}
f(t)=|t-t_0|^{k+1}\left(a_{k+1}+a_{k+2}(t-t_0)+\ldots+a_n(t-t_0)^{n-k-1}\right),
\end{align}
so we may choose a sufficiently small left neighborhood of $t_0$ such that $|f|=f$ on that left neighborhood and as $\phi_L(f)>\phi_L(-f)$, it follows that $\phi_L(|f|)=|\phi_L(f)|$. In case $k$ is even, we have that 
\begin{align}\label{equation for principal 2}
-f(t)=|t-t_0|^{k+1}\left(a_{k+1}+a_{k+2}(t-t_0)+\ldots+a_n(t-t_0)^{n-k-1}\right),
\end{align}
so $|f|=-f$ on a sufficiently small left neighborhood of $t_0$. Hence $\phi_L(f)<\phi_L(-f)$ and $\phi_L(|f|)=|\phi_L(f)|$. Suppose now that $f^{(k+1)}_-(t_0)=(k+1)!a_{k+1}<0$. If $k$ is odd, then similarly, we find that we may chose a sufficiently small left neighborhood of $t_0$ such that $|f|=-f$. Since $\phi_L(f)<\phi_L(-f)$, it follows that $\phi_L(|f|)=|\phi_L(f)|$. If $k$ is even, then there is a sufficiently small left neighborhood of $t_0$ such that $|f|=f$, and as $\phi_L(f)>\phi_L(-f)$, we see that $\phi_L(|f|)=|\phi_L(f)|$. Hence $\phi_L$ is a Riesz homomorphism with kernel $L_{t_0}$. A similar argument shows that $\varphi_R$ is a Riesz homomorphism with kernel $R_{t_0}$.
\end{proof}
For brevity, write $E:=\mathrm{PPol}^n([a,b])$. Let $t_0\in(a,b]$ and define $L_{t_0}^k$ for $1\le k\le n$ by
\begin{align}\label{E:L_t^k}
   L_{t_0}^k:=\left\{f\in M_{t_0}^E \colon f_-^{(k)}(t_0)=f_-^{(k-1)}(t_0)=\ldots=f_-'(t_0)=0\right\}, 
\end{align}
and similarly 
\begin{align}\label{E:R_t^k}
R_{t_0}^k:=\left\{f\in M_{t_0}^E \colon f_+^{(k)}(t_0)=f_+^{(k-1)}(t_0)=\ldots=f_+'(t_0)=0\right\}
\end{align}
for $1\le k\le n$ whenever $t_0\in[a,b)$. Note that $L_{t_0}^n=L_{t_0}$ and $R_{t_0}^n=R_{t_0}$. It turns out that all non-maximal prime ideals in $E$ are of this form.

\begin{theorem}\label{P: prime ideals in Pol^n}
The non-maximal prime ideals in $\mathrm{PPol}^n([a,b])$ are of the form $L_{t_0}^k$ for some $1\le k\le n$ and $t_0\in(a,b]$, or are of the form $R_{t_0}^k$ for some $1\le k\le n$ and $t_0\in[a,b)$.
\end{theorem}

\begin{proof}
Write $E:=\mathrm{PPol}^n([a,b])$. Let $t_0\in(a,b]$. Then $L_{t_0}^k$ is the preimage of the prime ideal 
\[
I_k:=\{x\in\mathbb{R}^{n+1}\colon x_i=0,\ 1\le i\le k\}
\]
under $\phi_L$ forcing it to be a prime ideal in $E$ as $\varphi_L$ is a Riesz homomorphism by Lemma~\ref{L: phi_L and phi_R are hom's}. Indeed, consider the linear map $\sigma_k\colon \mathbb{R}^{n+1}\to \mathbb{R}^k$ by $\sigma_k(x_1,\ldots,x_{n+1}):=(x_1,\ldots,x_k)$. If we equip $\mathbb{R}^{n+1}$ and $\mathbb{R}^k$ with the lexicographical ordering, it follows that $\sigma_k$ is a Riesz homomorphism. To see this, let $x\in\mathbb{R}^{n+1}$. If $x_1<-x_1$, then $|x|=-x$ and $\sigma_k(|x|)=|\sigma_k(x)|$, and if $x_1>-x_1$, then $|x|=x$ and $\sigma_k(|x|)=|\sigma_k(x)|$. Suppose that $x_i=0$ for $1\le i\le l<k$ and $x_{l+1}<-x_{l+1}$, then $|x|=-x$ and $\sigma_k(|x|)=|\sigma_k(x)|$, and if $x_{l+1}>-x_{l+1}$, then $|x|=x$ and $\sigma_k(|x|)=|\sigma_k(x)|$. In the case where $x_i=0$ for all $1\le i\le k$ it is clear that $\sigma_k(|x|)=|\sigma_k(x)|=0$. Hence, $\sigma_k$ is a Riesz homomorphism. Since the kernel of $\sigma_k$ is $I_k$, it follows that $I_k$ is an ideal in $\mathbb{R}^{n+1}$, and as the quotient vector lattice $\mathbb{R}^{n+1}/I_k$ is linearly ordered, we conclude that $I_k$ is a prime ideal in $\mathbb{R}^{n+1}$ by \cite[Theorem~33.2]{Zaanen}. Similarly, we conclude for $t_0\in [a,b)$ that $R_{t_0}^k$ is a prime ideal in $E$ for all $1\le k\le n$.  
Note that the linear maps $\phi_{k}\colon L_{t_0}^k\to \mathbb{R}$ given by $$\phi_{k}(f):=f_-^{(k+1)}(t_0)$$ have kernels $L_{t_0}^{k+1}$ for $1\le k\le n-1$, and the linear map $\phi_0\colon M_{t_0}^E\to \mathbb{R}$ given by $\phi_0(f):=f'_-(t_0)$ has kernel $L^1_{t_0}$. Hence 
\[
L_{t_0}^{n-1}/L_{t_0}\cong \dots \cong L^k_{t_0}/L_{t_0}^{k+1}\cong\dots\cong M_{t_0}^E/L^1_{t_0}\cong\mathbb{R}, 
\]
and similarly, we have 
\[
R_{t_0}^{n-1}/R_{t_0}\cong \dots \cong R^k_{t_0}/R_{t_0}^{k+1}\cong\dots\cong M_{t_0}^E/R^1_{t_0}\cong\mathbb{R}. 
\]
It follows that if $P$ is a non-maximal prime ideal in $E$, then it contains a minimal prime ideal which is of the form $L_{t_0}$ for some $t_0\in(a,b]$ or $R_{t_0}$ for some $t_0\in[a,b)$ by \Cref{L:minimal prime in poly}. Thus, we have $L_{t_0}\subseteq P\subsetneq M_{t_0}^E$ or $R_{t_0}\subseteq P\subsetneq M_{t_0}^E$. If $k\ge 1$ is the smallest number such that $L_{t_0}^k\subseteq P$, then we must have $L_{t_0}^k\subseteq P\subseteq L^{k-1}_{t_0}$, so $P=L_{t_0}^k$ or $P=L_{t_0}^{k-1}$. Similarly, if $P$ is a non-maximal prime ideal such that $R_{t_0}\subseteq P\subseteq M_{t_0}^E$, then $P=R_{t_0}$ or $P=R^k_{t_0}$ for some $k\ge 1$.
\end{proof}

\begin{corollary}
The vector lattice $\mathrm{PPol}^n([a,b])$ is prime Noetherian. Furthermore, any ascending chain of prime ideals is of length at most $n$, and it contains a chain of ascending prime ideals of length $n$.
\end{corollary}

\begin{proof}
The ascending chains of prime ideals
\[
L_{t_0}=L^n_{t_0}\subseteq L^{n-1}_{t_0}\subseteq\ldots\subseteq L_{t_0}^1\subseteq M_{t_0}^E\quad  \mbox{and}\quad R_{t_0}=R^n_{t_0}\subseteq R^{n-1}_{t_0}\subseteq\ldots\subseteq R_{t_0}^1\subseteq M_{t_0}^E 
\]
are of length $n$.
\end{proof}

For the piecewise polynomials of arbitrary degree $\mathrm{PPol}([a,b])$ and $t_0\in(a,b]$ we define the linear map $\psi_L\colon \mathrm{PPol}([a,b])\to c_{00}$
by 
\begin{align}\label{formula Riesz hom left2}
\psi_L(f):=(f(t_0),-f'_-(t_0),f_-^{''}(t_0),\ldots,(-1)^{k}f_-^{(k)}(t_0),(-1)^{k+1}f^{(k+1)}_+(t_0),\ldots).
\end{align}
Similarly for $t_0\in[a,b)$ we define the linear map $\psi_R\colon \mathrm{PPol}([a,b])\to c_{00}$   by 
\begin{align}\label{formula Riesz hom right2}
\psi_R(f):=(f(t_0),-f_+'(t_0),f_+^{''}(t_0),\ldots,(-1)^kf^{(k+1)}_+(t_0),(-1)^{k+1}f^{(k+1)}_+(t_0),\ldots). 
\end{align}
If we equip $c_{00}$ with the lexicographical ordering, then this yields a totally ordered vector lattice which will be denoted by $\mathrm{Lex}(\mathbb{N})$. For a literary reference, see \cite[Section~2]{Wortel}. It follows that $\psi_L$ and $\psi_R$ are Riesz homomorphisms.

\begin{lemma}\label{L: psi_L and psi_R are hom's}
For $t_0\in(a,b]$ the map $\psi_L$ as in \eqref{formula Riesz hom left2} is a Riesz homomorphism with kernel $L_{t_0}$ and for $t_0\in[a,b)$ the map $\psi_R$ as in \eqref{formula Riesz hom right2} is a Riesz homomorphism with kernel $R_{t_0}$. 
\end{lemma}

\begin{proof}
Note that $\psi_L$ is surjective as for any $(x_n)_{n\in\N}$ in $c_{00}$ there is a smallest $N\ge 1$ such that $x_n=0$ for all $n\ge N+1$, and the polynomial defined by 
\[
f(t):=x_1-x_2(t-t_0)+\ldots+(-1)^{N-1}\tfrac{1}{(N-1)!}x_{N}(t-t_0)^{N-1}
\]
satisfies $\psi_L(f)=(x_n)_{n\in\N}$. Let $f\in E$. Then it follows that $\psi_L(|f|)=|\psi_L(f)|$ via analogous reasoning for the map $\phi_L$ in $\eqref{formula Riesz hom left}$ where $n$ equals the degree of the polynomial that equals $f$ on a left neighborhood of $t_0$. Hence, $\psi_L$ is a Riesz homomorphism with kernel $L_{t_0}$. Analogous to the reasoning in proving that $\varphi_R$ is a Riesz homomorphism, it follows that $\psi_R$ is a Riesz homomorphism with kernel $R_{t_0}$.
\end{proof}

If we write $E:=\mathrm{PPol}([a,b])$, then as in \eqref{E:L_t^k} and \eqref{E:R_t^k} we consider
\[
L_{t_0}^k:=\left\{f\in M_{t_0}^E\colon f_-^{(k)}(t_0)=f_-^{(k-1)}(t_0)=\ldots=f'_-(t_0)=0\right\}
\]
for $t_0\in(a,b]$ and all $k\in \N$ in this case, and 
\[
R_{t_0}^k:=\left\{f\in M_{t_0}^E\colon f_+^{(k)}(t_0)=f_+^{(k-1)}(t_0)=\ldots=f'_+(t_0)=0\right\}
\]
for $t_0\in[a,b)$ and all $k\in \N$ in this case. Note that $\bigcap_{k=1}^\infty L_{t_0}^k=L_{t_0}$ and $\bigcap_{k=1}^\infty R_{t_0}^k=R_{t_0}$, and it follows that these are in fact all the non-maximal prime ideals in $E$.

\begin{theorem}\label{P: prime ideals in Pol}
The non-maximal and non-minimal prime ideals in $\mathrm{PPol}([a,b])$ are of the form $L_{t_0}^k$ for some $k\in\N$ and $t_0\in(a,b]$, or are of the form $R_{t_0}^k$ for some $k\in \N$ and $t_0\in[a,b)$.
\end{theorem}

\begin{proof}
Write $E:=\mathrm{PPol}([a,b])$. Let $t_0\in (a,b]$. Then $L^k_{t_0}$ is the preimage of the prime ideal 
\[
J_k:=\{x\in \mathrm{Lex}(\mathbb{N})\colon x_i=0,\ 1\le i\le k\} 
\]
under the Riesz homomorphism  $\psi_L$. It follows that $J_k$ is the kernel of $\tau_k\colon \mathrm{Lex}(\mathbb{N})\to\mathbb{R}^k$ defined by $\tau_k((x_n)_{n\in\N}):=(x_1,\ldots,x_k)$, which is proved to be a Riesz homomorphism analogously to showing that $\sigma_k$ is a Riesz homomorphism in \Cref{P: prime ideals in Pol^n}. Similarly, we have prime ideals 
\[
R_{t_0}^k:=\left\{f\in M_{t_0}^E\colon f_+^{(k)}(t_0)=f_+^{(k-1)}(t_0)=\ldots=f'_+(t_0)=0\right\}
\]
for all $k\in\N$. Furthermore, as in \Cref{P: prime ideals in Pol^n} we have $L^k_{t_0}/L^{k+1}_{t_0}\cong\mathbb{R}$ and $R^k_{t_0}/R^{k+1}_{t_0}\cong\mathbb{R}$ for $k\in\N$, and $M_{t_0}^E/L^1_{t_0}\cong\mathbb{R}$ and $M_{t_0}^E/R^1_{t_0}\cong\mathbb{R}$. 

Let $P$ be a non-maximal prime ideal in $E$ such that $L_{t_0}\subseteq P\subseteq M_{t_0}^E$ for some $t_0\in(a,b]$. If $P$ is contained in $L_{t_0}^k$ for all $k\in\N$, then $P\subseteq \bigcap_{k=1}^\infty L_{t_0}^k=L_{t_0}$, so $P=L_{t_0}$. On the other hand, if $k\in\N$ is the smallest number such that $L_{t_0}^k\subseteq P$, then $L_{t_0}^k\subseteq P\subseteq L^{k-1}_{t_0}$, so $P=L_{t_0}^k$ or $P=L_{t_0}^{k-1}$ as $L_{t_0}^{k-1}/L_{t_0}^{k}\cong\mathbb{R}$. Similarly, if $P$ is a non-maximal prime ideal such that $R_{t_0}\subseteq P\subseteq M_{t_0}^E$, then $P=R_{t_0}$ or $P=R^k_{t_0}$ for some $k\in\N$.
\end{proof}

\begin{corollary}
The vector lattice $\mathrm{PPol}([a,b])$ is prime Noetherian and contains ascending chains of prime ideals of arbitrary finite length.
\end{corollary}

\begin{proof}
Write $E:=\mathrm{PPol}([a,b])$. Let $P_1\subseteq P_2\subseteq\ldots$ be an ascending chain of prime ideals. Then for any $n\in\N$, it follows that $L_{t_0}\subseteq P_n\subseteq M^E_{t_0}$ for some $t_0\in(a,b]$ or $R_{t_0}\subseteq P_n\subseteq M^E_{t_0}$ for some $t_0\in [a,b)$. By \Cref{P: prime ideals in Pol} the only $P_n$ such that $L_{t_0}\subsetneq P_n\subsetneq M_{t_0}^E$ or $R_{t_0}\subsetneq P_n\subsetneq M_{t_0}^E$ must be of the form $P_n=L_{t_0}^k$ or $P_n=R_{t_0}^k$, respectively, for some $k\in\N$. In this case $P_n$ can only be contained in finitely many prime ideals, so the chain must be stationary. Note that for $n\in\N$ we have ascending chains of prime ideals
\[
L^n_{t_0}\subseteq L^{n-1}_{t_0}\subseteq\ldots\subseteq L_{t_0}^1\subseteq M_{t_0}^E\quad  \mbox{and}\quad R^n_{t_0}\subseteq R^{n-1}_{t_0}\subseteq\ldots\subseteq R_{t_0}^1\subseteq M_{t_0}^E 
\]
of length $n$ in $E$.
\end{proof}

In view of Theorem~\ref{Theorem: principal prime is maximal} there are no non-maximal principal prime ideals in uniformly complete Archimedean vector lattices. However, there is an abundance of principal prime ideals, even ascending chains of arbitrary finite length, in vector lattices of piecewise polynomials.

\begin{theorem}\label{T:principal primes in Pol}
Let $E$ be the vector lattice $\mathrm{PPol}^n([a,b])$ or $\mathrm{PPol}([a,b])$. Then all non-minimal prime ideals in $E$ are principal.
\end{theorem}

\begin{proof}
By Lemma~\ref{L:lemma about maximals}, every maximal ideal in $E$ is of the form $M_{t_0}^E$ for some $t_0\in[a,b]$, that is,
\[
M_{t_0}^E=\{f\in E\colon f(t_0)=0\}.
\]
We will show that $f(t):=|t-t_0|$ generates $M_{t_0}^E$. Indeed, suppose first that $t_0\in(a,b)$. If $g\in M_{t_0}^E$, then there is a left neighborhood $I_{t_0}$ and a right neighborhood $J_{t_0}$ of $t_0$ in $[a,b]$ such that $g$ is a polynomial on $I_{t_0}$ and $J_{t_0}$, and $|t-t_0|\le 1$ for all $t\in I_{t_0}\cup J_{t_0}$. That is, we have
\[
g(t)=a_1(t-t_0)+\ldots +a_n(t-t_0)^n\ \ \mbox{on $I_{t_0}$}\qquad \mbox{and}\qquad g(t)=b_1(t-t_0)+\ldots +b_n(t-t_0)^m\ \ \mbox{on $J_{t_0}$}.
\]
Hence, for $\mu:=\max\bigl\{\sum_{k=1}^n|a_k|,\sum_{k=1}^m|b_k|\bigr\}$, we find that $|g(t)|\le \mu|f(t)|$ for all $t\in I_{t_0}\cup J_{t_0}$. Outside $I_{t_0}\cup J_{t_0}$ we have that $f$ is strictly positive, so $gf^{-1}$ is bounded there and we conclude that $|g|\le \lambda f$ for some $\lambda>0$, showing that $M_{t_0}^E$ is generated by $f$. If $t_0$ is either $a$ or $b$, a similar argument using only a left or right neighborhood of $t_0$ shows that $M_{t_0}^E$ is generated by $f$ as well.  

Let $P$ be a non-maximal and non-minimal prime ideal in $E$. Then by \Cref{P: prime ideals in Pol^n} and \Cref{P: prime ideals in Pol} there is a $t_0\in[a,b]$ such that $P=L_{t_0}^k$ or $P=R_{t_0}^k$ for some $k\in\mathbb{N}$. Suppose first that $t_0\in(a,b)$. The functions
\[
f_l(t):=\begin{cases}
t_0-t&\mbox{if $a\le t\le t_0$}\\
\ \ 0&\mbox{if $t_0<t\le b$}
\end{cases}\qquad\mbox{and}\qquad f_r(t):=\begin{cases}
\ \ 0&\mbox{if $a\le t<t_0$}\\
t-t_0&\mbox{if $t_0\le t\le b$}
\end{cases}
\]
are disjoint, so it follows that either $f_l\in P$ or $f_r\in P$. Note that $P$ cannot contain both $f_l$ and $f_r$ as that would force $P$ to be the maximal ideal $M_{t_0}^E$ as discussed in the paragraph above. Suppose $f_l\in P$. Then $P=R_{t_0}^k$ and since each $f\in E$ can be written as 
\[
f(t)=a_0+a_1(t-t_0)+\ldots +a_n(t-t_0)^n 
\]
locally on a right neighborhood of $t_0$ for some $n\in\mathbb{N}$, it follows that $f\in P$ if and only if $a_0=\ldots=a_k=0$. Define $g_{k+1}(t):=|t-t_0|^{k+1}$ and note that for $f\in P$ there is a $\delta>0$ such that $|t-t_0|\le 1$ for all $t\in [t_0,t_0+\delta)$, and then 
\begin{align}\label{E:principal prime in Pol}
|f(t)|=|t-t_0|^{k+1}\left|a_{k+1}+a_{k+2}(t-t_0)+\ldots+a_n(t-t_0)^{n-k-1}\right|\le g_{k+1}(t)\sum_{i=k+1}^n|a_i|.
\end{align}
Since $g_{k+1}$ is strictly positive on $[t_0+\delta,b]$, it follows that $fg_{k+1}^{-1}$ is bounded there and so there is a $\lambda>0$ such that $|f|\le \lambda g_{k+1}$ on $[t_0,b]$. Since $P\subsetneq M_{t_0}^E$, there also is a $\mu>0$ such that $|f|\le \mu f_l$ on $[a,t_0]$, which implies that the function
\[
g(t):=\begin{cases}
f_l(t)&\mbox{if $a\le t\le t_0$}\\
g_{k+1}(t)&\mbox{if $t_0<t\le b$}
\end{cases}
\]
generates the ideal $P$. Similarly, in the case where $f_r\in P$ instead of $f_l$, it follows that  $P=L_{t_0}^k$ and the function
\[
h(t):=\begin{cases}
g_{k+1}(t)&\mbox{if $a\le t<t_0$}\\
f_r(t)&\mbox{if $t_0\le t\le b$}
\end{cases}
\]
generates the ideal $P$. Finally, if $t_0=a$ or $t_0=b$, then it follows from \eqref{E:principal prime in Pol} that $P$ is generated by either $(t-a)^{k+1}$ or $(b-t)^{k+1}$, respectively.
\end{proof}

Note that by Theorem~\ref{T: Cohen} the vector lattices of piecewise polynomials must contain non-principal prime ideals, which are precisely the minimal prime ideals as shown below.

\begin{corollary}\label{L:minimal not principal Ppol}
Let $E$ be the vector lattice $\mathrm{PPol}^n([a,b])$ or $\mathrm{PPol}([a,b])$. Then none of the minimal prime ideals in $E$ are principal.
\end{corollary}

\begin{proof}
For any minimal prime ideal $P$ in $E$ there is a $t_0\in[a,b]$ such that $P=L_{t_0}$ or $P=R_{t_0}$ by \Cref{L:minimal prime in poly}. For every positive function $g$ in $P$ there is a $\delta>0$ such that $g$ is zero on either $(t_0-\delta,t_0]$ or $[t_0,t_0+\delta)$. Hence, we can construct two disjoint non-zero functions supported in $(t_0-\delta,t_0]$ or we can construct two disjoint functions supported in $[t_0,t_0+\delta)$. Hence, neither $L_{t_0}$ nor $R_{t_0}$ can be principal. 
\end{proof}

{\it Acknowledgments.}
The first author acknowledges financial support from the Slovenian Research Agency, Grants No. P1-0222, J1-2453 and J1-2454. 

\bibliographystyle{alpha}
\bibliography{prime}

\begin{thebibliography}{Wor19}

\bibitem[AA02]{AbramovichAliprantis}
Y.~A. Abramovich and C.~D. Aliprantis.
\newblock {\em An invitation to operator theory}, volume~50 of {\em Graduate
  Studies in Mathematics}.
\newblock American Mathematical Society, Providence, RI, 2002.

\bibitem[AB06]{AliprantisBurkinshaw}
C.~D. Aliprantis and O.~Burkinshaw.
\newblock {\em Positive operators}.
\newblock Springer, Dordrecht, 2006.
\newblock Reprint of the 1985 original.

\bibitem[AT07]{Tourky07}
C.~D. Aliprantis and R.~Tourky.
\newblock {\em Cones and duality}, volume~84 of {\em Graduate Studies in
  Mathematics}.
\newblock American Mathematical Society, Providence, RI, 2007.

\bibitem[Coh50]{Cohen50}
I.~S. Cohen.
\newblock Commutative rings with restricted minimum condition.
\newblock {\em Duke Math. J.}, 17:27--42, 1950.

\bibitem[JK62]{Johnson-Kist}
D.~G. Johnson and J.~E. Kist.
\newblock Prime ideals in vector lattices.
\newblock {\em Canadian J. Math.}, 14:517--528, 1962.

\bibitem[Kan09]{Kaniuth}
E.~Kaniuth.
\newblock {\em A course in commutative {B}anach algebras}, volume 246 of {\em
  Graduate Texts in Mathematics}.
\newblock Springer, New York, 2009.

\bibitem[Kap49]{Kaplansky49}
I.~Kaplansky.
\newblock Elementary divisors and modules.
\newblock {\em Trans. Amer. Math. Soc.}, 66:464--491, 1949.

\bibitem[LZ71]{Zaanen}
W.~A.~J. Luxemburg and A.~C. Zaanen.
\newblock {\em Riesz spaces. {V}ol. {I}}.
\newblock North-Holland Publishing Co., Amsterdam-London; American Elsevier
  Publishing Co., New York, 1971.
\newblock North-Holland Mathematical Library.

\bibitem[Sch74]{Schaefer}
H.~H. Schaefer.
\newblock {\em Banach lattices and positive operators}.
\newblock Springer-Verlag, New York-Heidelberg, 1974.
\newblock Die Grundlehren der mathematischen Wissenschaften, Band 215.

\bibitem[Wor19]{Wortel}
M.~Wortel.
\newblock Lexicographic cones and the ordered projective tensor product.
\newblock In {\em Positivity and noncommutative analysis}, Trends Math., pages
  601--609. Birkh\"{a}user/Springer, Cham, 2019.

\bibitem[Yos42]{Yosida}
K.~Yosida.
\newblock On the representation of the vector lattice.
\newblock {\em Proc. Imp. Acad. Tokyo}, 18:339--342, 1942.

\end{thebibliography}

\end{document}